\documentclass[11pt]{amsart}
\usepackage{amsmath}
\usepackage{amssymb}
\usepackage{graphicx}
\usepackage{url}
\usepackage{multirow}

\usepackage{algorithm}
\usepackage{algorithmic}
\usepackage{float}
\usepackage{comment}
\usepackage[font=small,labelfont=bf]{caption}
\usepackage{subcaption}

\usepackage{color}
\usepackage{hyperref}
\usepackage{multicol}

\newtheorem{theorem}{Theorem}[section]

\newtheorem{lemma}[theorem]{Lemma}

\theoremstyle{definition}
\newtheorem{definition}[theorem]{Definition}

\newtheorem{problem}[theorem]{Problem}

\numberwithin{equation}{section}

\title[quasi-convex smooth optimization by a comparison oracle] {On quasi-convex smooth optimization problems by a comparison oracle}

\author[Gasnikov]{Alexander\,V.\,Gasnikov}
\address[A.\,V.\,Gasnikov]{Innopolis University, Moscow Institute of Physics and Technology, Caucasus Mathematical Center, and Steklov Mathematical Institute of Russian Academy of Sciences.}
\email{\tt gasnikov@yandex.ru}

\author[Alkousa]{Mohammad S. Alkousa}
\address[M.~S.~Alkousa]{Innopolis University, Moscow Institute of Physics and Technology}
\email{\tt mohammad.alkousa@phystech.edu}

\author[Lobanov]{Alexander\,V.\,Lobanov}
\address[A.\,V.\,Lobanov]{Moscow Institute of Physics and Technology, Skolkovo Institute of Science and Technology, and ISP RAS Research Center for Trusted Artificial Intelligence.}
\email{\tt lobbsasha@mail.ru}

\author[Dorn]{Yuriy\,V.\,Dorn}
\address[Y.\,V.\,Dorn]{Institute for Artificial Intelligence, Lomonosov Moscow State University, and Moscow Institute of Physics and Technology }
\email{\tt dornyv@my.msu.ru}

\author[Stonyakin]{Fedor~S.~Stonyakin}
\address[F.~S.~Stonyakin]{Moscow Institute of Physics and Technology, V. I. Vernadsky Crimean Federal University, and Innopolis University.}
\email{\tt fedyor@mail.ru}

\author[Kuruzov]{Ilya\,A.\,Kuruzov}
\address[I.\,A.\,Kuruzov]{Research Center for Artificial Intelligence,  and Innopolis University}
\email{\tt kuruzov2014@mail.ru}

\author[Singh]{Sanjeev\,R.\,Singh}
\address[S.\,R.\,Singh]{Times School of Media, Bennett University}
\email{\tt sanjeev.singh@bennett.edu.in}


\begin{document}

\begin{abstract}
Frequently, when dealing with many machine learning models, optimization problems appear to be challenging due to a limited understanding of the constructions and characterizations of the objective functions in these problems. Therefore, major complications arise when dealing with first-order algorithms, in which gradient computations are challenging or even impossible in various scenarios. For this reason, we resort to derivative-free methods (zeroth-order methods). This paper is devoted to an approach to minimizing quasi-convex functions using a recently proposed (in \cite{Zhang2024Comparisons}) comparison oracle only. This oracle compares function values at two points and tells which is larger, thus by the proposed approach, the comparisons are all we need to solve the optimization problem under consideration. The proposed algorithm to solve the considered problem is based on the technique of comparison-based gradient direction estimation and the comparison-based approximation normalized gradient descent. The normalized gradient descent algorithm is an adaptation of gradient descent, which updates according to the direction of the gradients, rather than the gradients themselves. We proved the convergence rate of the proposed algorithm when the objective function is smooth and strictly quasi-convex in $\mathbb{R}^n$,  this algorithm needs $\mathcal{O}\left( \left(n D^2/\varepsilon^2 \right) \log\left(n D / \varepsilon\right)\right)$ comparison queries to find an $\varepsilon$-approximate of the optimal solution, where $D$ is an upper bound of the distance between all generated iteration points and an optimal solution. 
\end{abstract}

\maketitle

\section{Introduction}
Optimization may be regarded as the cornerstone in the realm of machine learning and its applications. Frequently, when dealing with many machine learning models, optimization problems appear. Thus, we must use optimization algorithms to solve these problems, for instance, advancements in stochastic gradient descent (SGD) such as ADAM \cite{Diederik2015method}, Adagrad \cite{Duchi2011Adaptive}, etc., serve as foundational methods for the training of deep neural networks and many other tasks. However, many optimization problems with challenges appear due to a limited understanding of the constructions and characterizations of the objective functions in these problems, such as black-box adversarial attacks on neural networks \cite{Chen2017Zeroth,Madry2018Towards,Papernot2017Practical} and policy search in reinforcement learning \cite{Salimans2017Evolution,Choromanski2018Optimizing}. Therefore, major complications arise when dealing with first-order algorithms, in which gradient computations are challenging or even impossible in various scenarios. For this reason, we resort to derivative-free methods (zeroth-order methods). 

Derivative-free or zeroth-order optimization is one of the oldest areas in optimization \cite{Rosenbrock1960,Fabian1967,Brent1973,Conn2009}, which constantly attracts attention from the learning community, mostly in connection to online learning in the bandit setup \cite{Bubeck2012Regret}, deep learning \cite{Chen2017Zeroth,Gao2018Black}, federated learning \cite{Dai2020Federated,Alashqar2023,Patel2022},  overparameterized models \cite{Lobanov2023Accelerated}, online optimization \cite{Bach2016Highly,Akhavan2022gradient}, and control system performance optimization \cite{Bansal2017Goal,Xu2022Violation}.

Derivative-free optimization methods with function evaluations have gained prominence, with a provable guarantee for convex optimization \cite{Duchi2015Optimal,Nesterov2017Random} and non-convex optimization \cite{Fang2018Near,Ji2019Improved,Zhang2022Zeroth,Balasubramanian2022Zeroth}. Among the literature for these methods, direct search methods \cite{Kolda2003Optimization} proceed by comparing function values, including the directional direct search method \cite{Audet2006Mesh} and the Nelder-Mead method \cite{Nelder1965simplex} as examples. However, the directional direct search method does not have a known rate of convergence, meanwhile, the Nelson-Mead method may fail to converge to a stationary point for smooth functions \cite{Dennis1991Direct}.

Recently, special attention has been paid to zeroth-order algorithms with so-called \textit{Comparioson Oracle} (or \textit{Order Oracle}) \cite{Zhang2024Comparisons,Lobanov2024Acceleration}. This oracle compares function values at two points and tells which is larger, thus by algorithms using this type of oracle, the comparisons are all we need to solve the optimization problem. A relevant result in \cite{Bergou2020Stochastic}, where it was proposed the stochastic three points method (STP) which uses, in $\mathbb{R}^n$, an oracle that compares three function values at once and achieves the oracle complexity $\widetilde{\mathcal{O}}\left( n / \varepsilon \right)$ in the smooth convex case ($\widetilde{\mathcal{O}}(\cdot)$ used to hide the logarithmic factors), and $\mathcal{O}\left( (n L/ \mu)\log \left(1/\varepsilon\right) \right)$ in the $L$-smooth $\mu$-strongly convex case. A little later, the authors in \cite{Gorbunov2020stochastic} proposed a version of STM with momentum and modified STP to the case of importance sampling, improving the oracle complexity estimate in the strongly convex case  $\mathcal{O}\left( \left(\widehat{L}/\mu\right) \log \left(1/\varepsilon\right) \right)$, where $\widehat{L}$ is a constant connected with Lipchitz constants $L_i$ of the i-th coordinate gradient, for $i = 1, \ldots, n$, of the objective function.  In \cite{Lobanov2024Acceleration}, the authors, by using the comparison oracle, proposed a new approach to create non-accelerated optimization algorithms in non-convex, convex, and strongly convex settings with oracle complexity $\widetilde{\mathcal{O}}\left( \widehat{L} F_0 / \varepsilon^2 \right)$, $\widetilde{\mathcal{O}}\left(\widehat{L} R^2 / \varepsilon \right)$ and $\widetilde{\mathcal{O}}\left( \left(\widehat{L}/\mu\right) \log \left(1/\varepsilon\right) \right)$ respectively, where $F_0 = f(x_0) - f(x_*)$, $R$ characterizes the distance between the initial point $x_0$ and the optimal solution $x_*$. Also, by using the comparison oracle, in \cite{Zhang2024Comparisons} several algorithms have been proposed for $L$-smooth functions with complexity $\mathcal{O}\left( \left(n LD^2 / \varepsilon\right) \log \left(n LD^2 / \varepsilon\right)\right)$ or $\mathcal{O}\left(n^2 \log\left(n LD^2 / \varepsilon\right)\right)$, where $D$ is an upper bound of the distance between all generated iteration points and an optimal solution. For $L$-smooth (possibly non-convex) functions it was proposed an algorithm with success probability at least $2/3$ to find an $\varepsilon$-approximate solution with complexity $\mathcal{O}\left( \left(n L/ \varepsilon^2\right) \log n \right)$. We mention here that to the best of our knowledge, by using the comparison oracle there is not any approach or a specific algorithm focusing on the class of quasi-convex functions. 

The study of quasi-convex functions is several decades old \cite{Fenchel1953,Luenberger1968,Nikaido1954}. Optimization problems with quasi-convex functions arise in numerous fields, including spanning economics \cite{Varian1985,Laffont2009}, control \cite{Seiler2010}, industrial organization \cite{Wolfstetter1999}, computer vision \cite{Kanade2006,Kanade2007}, computational geometry \cite{Eppstein2005}, and many other important fields of science. While quasi-convex optimization problems have many applications, it remains difficult for non-experts to specify and solve them in practice (see \cite{Agrawal2020Disciplined}).  For solving quasi-convex optimization problems there exist many works. A pioneering paper by Nesterov \cite{Nesterov1984Minimization}, was the first to suggest an efficient algorithm, namely \textit{Normalized Gradient Descent}, and prove that this algorithm attains an $\varepsilon$-optimal solution within $O\left(1/\varepsilon^2\right)$ iterations for differentiable quasi-convex functions. This work was later extended in \cite{Kiwiel2001}, showing that the same result may be achieved assuming upper semi-continuous quasi-convex objectives. In \cite{Konnov2003} it was shown how to attain faster rates for quasi-convex optimization, but with knowing the optimal value of the objective, which generally does not hold in practice.

In this paper, we proposed an approach to minimizing quasi-convex functions using a recently proposed comparison oracle only \cite{Zhang2024Comparisons}, thus by the proposed approach, the comparisons are all we need to solve the optimization problem under consideration. The proposed algorithm is based on the technique of comparison-based gradient direction estimation and the comparison-based approximation normalized gradient descent. The normalized gradient descent algorithm is an adaptation of gradient
descent, which updates according to the direction of the gradients, rather than the gradients themselves. We proved the convergence rate of the proposed algorithm when the objective function is smooth and strictly quasi-convex in $\mathbb{R}^n$,  this algorithm needs $\mathcal{O}\left( \left(n D^2/\varepsilon^2 \right) \log\left(n D / \varepsilon\right)\right)$ comparison queries to find an $\varepsilon$-approximation of the optimal solution. 

We mention here that, there are many interesting questions we left for future works, such as the question about the possibility of improving the concluded complexity in this paper. Also,  what is about the comparison oracle for the class of $(L_0, L_1)$-smooth functions \cite{gorbunov2024methods,vankov2024optimizing}, which generalizes the class of $L$-smooth functions? and for non-smooth functions, or for weakly smooth functions? In addition to the previous questions, we plan to test the proposed approach by comparing it with other methods for simple examples and then for an applied real task in (for example) machine learning. 

The paper consists of an introduction and 3 main sections. In Sect.~\ref{sect_basics} we mention the problem statements with the basic facts, tools, and definitions that will be useful in the next sections of the paper. Sect.~\ref{section_EGDbyComp} is devoted to the scheme of estimation of gradient direction using the comparison oracle.  In Sect.~\ref{sect:quasicompamethod} we propose a comparison-based approximation adaptive normalized gradient descent algorithm and analyze its convergence for solving the problem under consideration.

\section{Fundamentals and basic definitions}\label{sect_basics}

Let $f: \mathbb{R}^n \longrightarrow \mathbb{R}$, an $L$-smooth function, this means that it holds the following inequality 
\begin{align*}
\|\nabla f(x)-\nabla f(y) \|_* \leq L \| x - y\| \quad \forall \, x, y \in \mathbb{R}^n, 
\end{align*}
where $\|\cdot\|_*$ is a conjugate norm of $\|\cdot\|$.  In this paper, we consider the following optimization problem
\begin{equation}\label{main_problem}
\min\limits_{x\in \mathbb{R}^n} f(x). 
\end{equation}

We denote by $x_*$ to a solution of the problem \eqref{main_problem}, and by $f^* = f(x_*)$ to the optimal value of the objective function $f$. 

\begin{definition}[Quasi-Convexity]\label{def:quasiconvex}
	A common definition is as follows. We say that a function $f: \mathbb{R}^n \longrightarrow \mathbb{R}$ is quasi-convex if any $\alpha$-sublevel-set of $f$ is convex, i.e., for any $\alpha \in \mathbb{R}$, the set $\mathcal{L}_{\alpha}(f) : = \{x \in \mathbb{R}^n: f(x) \leq \alpha \}$ is convex. 
	
	We mention another equivalent definition as follows. We say that a function $f$ is quasi-convex, if for any $x, y \in \mathbb{R}^n$ such that $f(y) \leq f(x)$, it follows that 
	\[
	\langle \nabla f(x) , y - x \rangle \leq 0. 
	\]
	
	The equivalence between these two definitions can be found in \cite{boyd2004convex}. We further say that f is \textit{strictly-quasi-convex} \cite{Hazan2015Beyond}, if it is quasi-convex and its gradients vanish only at the global minima, i.e.,
	\[
	\forall y \in \mathbb{R}^n,\;\;  f(y)>  f^* \quad \Longrightarrow \quad \|\nabla f(y)\|_* > 0.
	\]
	In other words, 
	\[
	\exists \gamma >0, \quad \text{such that} \quad \|\nabla f(y)\|_* \geq \gamma > 0, \quad \forall y \in \mathbb{R}^n.  
	\]
\end{definition}

Following \cite{Nesterov_book} (see Subsec. 3.2.2), let $f$ is a given function and let us fix some $y \in \mathbb{R}^n$, for each gradient $ \nabla f (x)$ at the point $x \in \mathbb{R}^n$, we introduce the following auxiliary value
\begin{equation}\label{eq:vfDef}
v_f(x, y)= 
\begin{cases}
\left\langle\frac{\nabla f(x)}{\|\nabla f(x)\|_{*}},x-y\right\rangle; & \nabla f(x) \ne \textbf{0},\\
0;  & \nabla f(x) = \textbf{0},
\end{cases}
\end{equation}
where the bold $\textbf{0}$ denotes the vector of elements $0$ in $\mathbb{R}^n$.

As shown in \cite{Nesterov_book}, for $y = x_*$, the value $v_f(x, x_*)$ can be equal to the length of the projection of $x_*$ onto the hyperplane orthogonal to the vector $\nabla f(x)$. If a numerical method can produce a point $x_k$ for which $v_f(x_k, x_*) \leq \varepsilon$ for some small enough $\varepsilon > 0$, then, for this point, under the assumption that the function is continuous, we can guarantee the achieving a suitable quality of the solution. Here we can consider not only convex but also quasi-convex problems.

\begin{lemma}[\cite{Nesterov_book}, Theorem 1.5.5] 
	Let  $f(x)$ be a quasi-convex, and let  $f(x) \geq f(x_*)$. Define the function
	\begin{equation*}
	\omega(\tau) = \max\limits_{x\in \mathbb{R}^n } \{f(x) - f(x_*) : \|x-x_*\| \leq \tau \},
	\end{equation*}
	where $\tau$ is a positive number. Then, for any $y \in \mathbb{R}^n$ it holds the following
	\begin{equation*}
	f(y) - f(x_*) \leq \omega(v_f(y,x_*)).
	\end{equation*}
\end{lemma}

Since in this paper, we study the optimization of smooth functions using comparisons, we mention the following main definition. 

\begin{definition}[Comparison oracle \cite{Zhang2024Comparisons}]
	For a function $f$, we define the \emph{comparison oracle} of $f$ as $O_f^{\text{comp}}:  \mathbb{R}^n  \times \mathbb{R}^n \longrightarrow \{-1,1\},$ such that
	\begin{align}\label{eq:comparison_oracle}
	O_{f}^{\text{comp}}(x, y)=
	\begin{cases}
	1;  & \text{if \; $f(x) \geq f(y)$,} \\
	-1; & \text{if \; $f(x)\leq f(y)$.} 
	\end{cases}
	\end{align}
	When $f(x) = f(y)$, we can take $1$ or $-1$ as an output of this oracle.   
\end{definition}

\section{Estimation of gradient direction by comparisons}\label{section_EGDbyComp}

In \cite{Zhang2024Comparisons}, it was shown that for a given point $ x  \in \mathbb{R}^n$ and a direction $ v\in \mathbb{R}^n$, we can use one comparison query to understand whether the inner product $ \langle \nabla f(x), v \rangle$ is roughly positive or negative. This inner product determines whether $x+v$ is following or against the direction of $\nabla f(x)$, also known as \emph{directional preference} (DP) in~\cite{karabag2021smooth}.

In \cite{Zhang2024Comparisons}, to estimate a gradient direction of a smooth function $f$ by using a comparison oracle only, the authors proposed an algorithm, called Comparison-based Gradient Direction Estimation, listed as Algorithm \ref{alg:comparison-GDE} below. In each iteration of Algorithm \ref{alg:comparison-GDE}, we need to determine a directional preference by Algorithm \ref{alg:DP}. Let us denote $\mathbb{B}_r(x) : = \left\{ y \in \mathbb{R}^n: \|y - x\| \leq r \right\}$. 

\begin{algorithm}[htp]
	\caption{Directional Preference (DP($x, v, \Delta$))  \cite{Zhang2024Comparisons}.}
	\label{alg:DP}
	\begin{algorithmic}[1]
		\REQUIRE Comparison oracle $O_{f}^{\text{comp}}$ of $f : \mathbb{R}^n \longrightarrow  \mathbb{R}$, $ x \in \mathbb{R}^n$, unit vector $ v \in \mathbb{B}_1(0)$, $\Delta>0$ a precision for directional preference. 
		\IF{$O_{f}^{\text{comp}} \left( x+\frac{2 \Delta}{L} v, x\right)=1$}
		\STATE Return ''$\langle \nabla f(x), v \rangle \geq - \Delta$'',
		\ELSE   
		\STATE (i.e., if $O_{f}^{\text{comp}}\left(x+\frac{2\Delta}{L} v, x\right)=-1$)
		\STATE Return ''$\langle \nabla f(x), v \rangle \leq  \Delta$''.
		\ENDIF
	\end{algorithmic}
\end{algorithm}

\begin{algorithm}[htp]
	\caption{Comparison-based Gradient Direction Estimation (Comparison-GDE($x,\delta,\gamma$)) \cite{Zhang2024Comparisons}.}\label{alg:comparison-GDE}
	\begin{algorithmic}[1]
		\REQUIRE Comparison oracle $O_{f}^{\text{comp}}$ of $f: \mathbb{R}^n \longrightarrow \mathbb{R}$, precision $\delta$, lower bound $\gamma$ on $\|\nabla f(x)\|_*$.   
		\STATE Set $\Delta := \frac{\delta \gamma}{4n^{3/2}} $. Denote $\nabla f( x) = (g_1,\ldots,g_n)^{\top}$. 
		\STATE Call Algorithm \ref{alg:DP} with inputs $( x, e_1, \Delta), \ldots, (x,e_n,\Delta)$ where $e_i$ is the $i^{\text{th}}$ standard basis with $i^{\text{th}}$ coordinate being 1 and others being 0. This determines whether $g_i\geq-\Delta$ or $g_i\leq\Delta$ for each $i\in \{1,2, \ldots, n\}$. Without loss of generality, we assume 
		\begin{equation*}
		g_i \geq -\Delta, \quad\forall i \in \{1, \ldots, n\}.
		\end{equation*}
		Otherwise, take a minus sign for the $i^{\text{th}}$ coordinate.
		\STATE We next find the approximate largest one among $g_1,\ldots,g_n$. Call  DP$\left(x,\frac{1}{\sqrt{2}}(e_1-e_2), \Delta\right)$ (Algorithm \ref{alg:DP}). This determines whether $g_1\geq g_2-\sqrt{2}\Delta$ or $g_2\geq g_1-\sqrt{2}\Delta$. If the former, call DP$\left(x,\frac{1}{\sqrt{2}}(e_1-e_3),\Delta\right)$. If the later, call DP$\left(x,\frac{1}{\sqrt{2}}(e_2-e_3),\Delta\right)$. Iterate this until $e_n$, we find the $i^{*}\in \{1, \ldots, n\}$ such that
		\begin{equation*}
		g_{i^*}\geq\max\limits_{1 \leq i \leq n} \left\{g_i-\sqrt{2}\Delta\right\}. 
		\end{equation*}
		\FOR{$i=1, 2 \ldots, n$ (except $i=i^{*}$)}
		\STATE  Initialize $\alpha_i\leftarrow 1/2$.
		\STATE Apply binary search to $\alpha_i$ in $\left\lceil \log_2\left( \frac{\gamma}{\Delta}\right)+1 \right\rceil$ iterations by calling DP$\left(x,\frac{\alpha_{i} e_{i^{*}}- e_i}{\sqrt{1+\alpha_{i}^{2}}}, \Delta\right)$. For the first iteration with $\alpha_i=\frac{1}{2}$, if $\alpha_i g_{i^*}-g_i\geq -\sqrt{2}\Delta$ we then take $\alpha_i=\frac{3}{4}$; if $\alpha_i g_{i^*}-g_i\leq\sqrt{2}\Delta$ we then take $\alpha_i=\frac{1}{4}$. Later iterations are similar. Upon finishing the binary search, $\alpha_i$ satisfies
		\begin{equation*}
		g_i-\sqrt{2}\Delta\leq \alpha_i g_{i^*}\leq g_i+\sqrt{2}\Delta.
		\end{equation*}
		\ENDFOR
		\STATE \textbf{Output:} $\widetilde{g}(x) = \frac{\alpha}{\|\alpha\|}$ where $\alpha=(\alpha_1,\ldots,\alpha_n)^{\top}$, $\alpha_i$ ($i\neq i^{*}$) is the output of the for loop, $\alpha_{i^{*}}=1.$
	\end{algorithmic}
\end{algorithm}

For Algorithm \ref{alg:comparison-GDE}, in \cite{Zhang2024Comparisons}, it was proved the following result. 
\begin{theorem}\label{thm:Comparison-GDE}
	For an $L$-smooth function $f: \mathbb{R}^n \longrightarrow \mathbb{R}$ and a point $ x \in \mathbb{R}^n$. If we are given a parameter $\gamma>0$ such that $\|\nabla f(x)\|_*\geq \gamma$, then by Algorithm \ref{alg:comparison-GDE} we obtain (as an output) an estimate $\widetilde{g}(x)$ of the direction of $\nabla f(x)$ using $O\left(n\log(n/\delta)\right)$ queries to the comparison oracle $O_{f}^{\text{comp}}$ of $f$ \eqref{eq:comparison_oracle}, which satisfies
	\begin{align*}
	\left\|\widetilde{g}(x)-\frac{\nabla f(x)}{\|\nabla f(x) \|_* }\right\|_*\leq\delta.
	\end{align*}
\end{theorem}

\section{Quasi-convex optimization by comparisons}\label{sect:quasicompamethod}

In this section, we study the problem \eqref{main_problem},  when the objective function $f$ is a quasi-convex, using the comparison oracle \eqref{eq:comparison_oracle}. Thus, let us formulate the problem under consideration \eqref{main_problem} in the following formulation. 
\begin{problem}[comparisons-based quasi-convex optimization]\label{main_problem2}
	In this problem, we are given query access to a comparison oracle $O_{f}^{\text{comp}}$ for an $L$-smooth strictly quasi-convex function $f: \mathbb{R}^n \longrightarrow \mathbb{R}$ whose minimum is achieved at a point $x_*$. The goal is solving the problem \eqref{main_problem} in the following sense: 
	\begin{center}
		For a given $\varepsilon>0$, find a point $\widehat{x}$ such that $v_{f} (\widehat{x}, x_*) \leq \varepsilon.$
	\end{center}
\end{problem}

For Problem \ref{main_problem2}, we will propose an algorithm (see Algorithm \ref{alg:quasi}), which applies Comparison-GDE (Algorithm \ref{alg:comparison-GDE}) with estimated gradient direction at each iteration to the Normalized Gradient Descent (NGD).

\begin{algorithm}[htp]
	\caption{Comparison based Approximation Adaptive Normalized  Gradient Descent (Comparison-AdaNGD).}
	\label{alg:quasi}
	\begin{algorithmic}[1]
		\REQUIRE $\varepsilon>0$, initial point  $x_1 \in \mathbb{R}^n$, $D>0$ such that $\|x - x_*\| \leq D; \; \forall x \in \mathbb{R}^n$. 
		\STATE Set $N = \frac{18 D^2}{\varepsilon^2}, \delta =\frac{\varepsilon}{2 D}, \gamma : = \varepsilon$ (a lower bound for $\|\nabla f(x)\|_*, \; \forall x \in \mathbb{R}^n$).
		\FOR{$k = 1, 2, \ldots, N$}
		\STATE $\widehat{g}_k \longleftarrow$ Comparison-GDE$(x_k, \delta, \gamma)$,
		\STATE $h_k = \frac{D}{\sqrt{2 k}}$,
		\STATE $x_{k+1} = x_k - h_k \widehat{g}_k. $
		\ENDFOR
	\end{algorithmic}
\end{algorithm}

To analyze Algorithm \ref{alg:quasi}, let us, at first, propose an algorithm called \textit{Approximation Normalized Adaptive Gradient Descent} (see Algorithm \ref{alg:ApproxAdap}) and analyze its convergence. 

\begin{algorithm}[htp]
	\caption{Approximation Adaptive Normalized  Gradient Descent (ApproxAdapt-NGD).}
	\label{alg:ApproxAdap}
	\begin{algorithmic}[1]
		\REQUIRE Number of iterations $N \geq 1$, initial point  $x_1 \in \mathbb{R}^n$, $D>0$ such that $\|x - x_*\| \leq D; \; \forall x \in \mathbb{R}^n$. 
		\FOR{$k = 1, 2, \ldots, N$}
		\STATE Calculate an estimate $\widetilde{g}_k$ of $\nabla f(x_k)$,
		\STATE $h_k = \frac{D}{\sqrt{2 k}}$,
		\STATE $x_{k+1} = x_k - h_k \widetilde{g}_k. $
		\ENDFOR
	\end{algorithmic}
\end{algorithm}

For Algorithm \ref{alg:ApproxAdap}, we have the following result. 

\begin{lemma}\label{lem:alg-ApproxAdap}
	Let $f: \mathbb{R}^n \longrightarrow \mathbb{R}$ be an $L$-smooth function. If at each iteration $k \geq 1$ of Algorithm \ref{alg:ApproxAdap}
	\begin{equation}\label{eq:1}
	\|\widetilde{g}_k\|_* = 1, \quad \text{and} \quad   \left\| \widetilde{g}_k - \frac{\nabla f(x_k)}{\|\nabla f(x_k)\|_*} \right\|_* \leq \delta,
	\end{equation}
	for $0 < \delta \leq 1 $. Then by Algorithm \ref{alg:ApproxAdap} we obtain a sequence $\{x_1, x_2 , \ldots, x_N\}$, such that
	\[
	\min\limits_{1 \leq k \leq N} v_f(x_k, x_*) \leq \frac{3 D}{\sqrt{2N}} + \delta D. 
	\]
\end{lemma}
\begin{proof}
The proof will be roughly similar (not quite as close) to the proof of Theorem 1.1 in \cite{levy2017online}. 
For any $k = 1, \ldots, N$, we have 
\[
\|x_{k+1} - x_*\|^2 \leq \|x_k - x_*\|^2 - 2 h_k \left\langle \widetilde{g}_k, x_k - x_* \right\rangle + h_k^2 \|\widetilde{g}_k\|_*^2.
\]

Since $\|\widetilde{g}_k\|_* = 1 \, \forall k \geq 1$, we have 
\[
\left\langle \widetilde{g}_k, x_k - x_*\right \rangle \leq \frac{1}{2h_k} \left( \|x_k - x_*\|^2  -  \|x_{k+1} - x_*\|^2 \right) + \frac{h_k}{2}. 
\]
Thus, we get
\begin{align*}
\left\langle \frac{\nabla f(x_k)}{\|\nabla f(x_k)\|_*}, x_k - x_* \right\rangle & 
\leq  \frac{1}{2h_k} \left( \|x_k - x_*\|^2  -  \|x_{k+1} - x_*\|^2 \right) + \frac{h_k}{2}+ 
\\& \quad +  \left\langle \frac{\nabla f(x_k)}{\|\nabla f(x_k)\|_*} - \widetilde{g}_k, x_k - x_* \right\rangle
\\& \leq \frac{1}{2h_k} \left( \|x_k - x_*\|^2  -  \|x_{k+1} - x_*\|^2 \right) + \frac{h_k}{2}  + \delta D.
\end{align*}
where in the last inequality we used the Cauchy-Schwartz inequality,  Eq.~\eqref{eq:1}, and the fact that $\|x - x_*\| \leq D$ for any $x \in \mathbb{R}^n$. Hence we have the following 
\begin{equation}\label{eq:2}
v_f(x_k, x_*) \leq  \frac{1}{2h_k} \left( \|x_k - x_*\|^2  -  \|x_{k+1} - x_*\|^2 \right) + \frac{h_k}{2}  + \delta D. 
\end{equation}

By taking the summation of both sides of Eq.~\eqref{eq:2} over $k = 1, \ldots, N$, we get the following 
\begin{equation*}
\sum_{k = 1}^{N}v_f(x_k, x_*) \leq \frac{1}{2}  \sum_{k = 1}^{N} \frac{1}{h_k} \left( \|x_k - x_*\|^2  -  \|x_{k+1} - x_*\|^2 \right) + \frac{1}{2} \sum_{k = 1}^{N} h_k  + \delta D N. 
\end{equation*}

Thus, 
\begin{align*}
N \cdot \min\limits_{1 \leq k \leq N} v_f(x_k, x_*) & \leq \sum_{k = 1}^{N}v_f(x_k, x_*)
\\& \leq \frac{1}{2}\sum_{k = 1}^{N} \|x_k - x_*\|^2 \left(\frac{1}{h_k} - \frac{1}{h_{k-1}}\right) + \frac{1}{2 h_N} \|x_{N+1} - x_*\|^2
\\& \quad  + \frac{1}{2} \sum_{k = 1}^{N} h_k + \delta D N.
\end{align*}
Where we denote $h_0 = \infty$. Therefore, we get 
\begin{align*}
N \cdot \min\limits_{1 \leq k \leq N} v_f(x_k, x_*) & \leq \frac{D^2}{2}\sum_{k = 1}^{N} \left(\frac{1}{h_k} - \frac{1}{h_{k-1}}\right) + \frac{D^2}{2 h_N}+ \frac{1}{2} \sum_{k = 1}^{N} h_k + \delta D N
\\& = \frac{D^2}{h_N}  + \frac{1}{2} \sum_{k = 1}^{N} h_k + \delta D N. 
\\& = D \sqrt{2N} + \frac{D}{2 \sqrt{2}} \sum_{k = 1}^{N} \frac{1}{\sqrt{k}} + \delta D N. 
\end{align*}

But $\sum\limits_{k = 1}^{N} \frac{1}{\sqrt{k}} \leq 2 \sqrt{N}$. Thus, we find
\begin{align*}
N \cdot \min\limits_{1 \leq k \leq N} v_f(x_k, x_*) & \leq D\sqrt{2N} + \frac{D \sqrt{N}}{\sqrt{2}} +  \delta D N = \frac{3D\sqrt{N}}{\sqrt{2}} +  \delta D N.
\end{align*}

Therefore, we conclude the following inequality
\[
\min\limits_{1 \leq k \leq N} v_f(x_k, x_*) \leq \frac{3 D}{\sqrt{2N}} + \delta D,
\]
which is the desired result.
\end{proof}

Now, let us prove the following lemma, which will be useful in the proof of the main theorem which explains the complexity of Algorithm \ref{alg:quasi}. 

\begin{lemma}\label{lemm3}
	Let $f: \mathbb{R}^n \longrightarrow \mathbb{R}$ be an $L$-smooth function, $D>0$ such that $\|x - x_*\| \leq D; \; \forall x \in \mathbb{R}^n$. For any given small enough $\varepsilon >0$, such that $\delta = \frac{\varepsilon}{2 D} \leq 1$, by Algorithm \ref{alg:quasi} we obtain a sequence $\{x_1, x_2, \ldots, x_N\}$, such that
	\begin{equation}\label{eq:solalg3}
	\min\limits_{1 \leq k \leq N} v_f(x_k, x_*)  \leq \varepsilon.
	\end{equation}
\end{lemma}
\begin{proof}
From line 2 in Algorithm \ref{alg:quasi}, we have $\|\nabla f(x)\|_* > \varepsilon, \; \forall x \in \mathbb{R}^n$. Thus, by Comparison-GDE (Algorithm \ref{alg:comparison-GDE}), we find that $\|\widehat{g}_k\|_* = 1, \; \forall k \geq 1,$ and $\left\| \widehat{g}_k - \frac{\nabla f(x_k)}{\|\nabla f(x_k)\|_*} \right\|_* \leq \delta \leq 1, \forall k \geq 1$ (see Theorem \ref{thm:Comparison-GDE}). Hence, from Lemma \ref{lem:alg-ApproxAdap}, we get the following inequality
\[
\min\limits_{1 \leq k \leq N} v_f(x_k, x_*) \leq\frac{3 D}{\sqrt{2N}} + \delta D. 
\]

Because, in Algorithm \ref{alg:quasi}, we have $N = \frac{18 D^2}{\varepsilon^2}$ and $\delta = \frac{\varepsilon}{2 D}$, we conclude the desired inequality Eq.~\eqref{eq:solalg3}. 
\end{proof}

\bigskip
For Algorithm \ref{alg:quasi}, we have the following result.

\begin{theorem}\label{theo:rate_quasi}
	Let $f: \mathbb{R}^n \longrightarrow \mathbb{R}$ be an $L$-smooth and strictly quasi-convex function. Let $\varepsilon > 0$. By Algorithm \ref{alg:quasi}, using $N = \mathcal{O}\left(\frac{D^2}{\varepsilon^2}\right)$ iterations and $\mathcal{O}\left(n \log\left(\frac{nD}{\varepsilon}\right)\right)$ calling of Algorithm \ref{alg:comparison-GDE} each iteration, i.e., using
	\[
	\mathcal{O}\left(\frac{nD^2}{\varepsilon^2} \log\left(\frac{n\delta}{\varepsilon}\right)\right)  
	\]
	queries, we obtain a sequence of points $\{x_1, x_2 , \ldots, x_N\}$, such that
	\[
	\min\limits_{1 \leq k \leq N} v_f(x_k, x_*) \leq \varepsilon. 
	\]
\end{theorem}
\begin{proof}
Since $f$ is a strictly quasi-convex function, there is $\gamma : = \varepsilon > 0$ (arbitrarily small enough), such that $\|\nabla f(x)\|_* \geq \gamma, \; \forall x \in \mathbb{R}^n$ (see Definition \ref{def:quasiconvex}). Thus, by Theorem \ref{thm:Comparison-GDE}, we can guarantee that
\[
\left\|\widehat{g}_k - \frac{\nabla f(x_k)}{\|\nabla f(x_k)\|_*} \right\|_* \leq \delta \leq 1, \quad \text{and} \quad \|\widehat{g}_k\|_* = 1. 
\]

With these approximate gradient directions $\widehat{g}_k \, (k = 1, 2, \ldots, N)$, by Lemma \ref{lemm3}, we get the following inequality
\[
\min\limits_{1 \leq k \leq N} v_f(x_k, x_*) \leq \varepsilon,
\]
and the query complexity of Algorithm \ref{alg:quasi} (the total comparison oracle calls) is equal to 
\[
N \cdot \mathcal{O}\left(n \log\left(\frac{n}{\delta}\right)\right) = \frac{18D^2}{\varepsilon^2} \cdot \mathcal{O}\left(n \log\left(\frac{2 n D}{\varepsilon}\right)\right) =  \mathcal{O}\left(\frac{nD^2}{\varepsilon^2} \log\left(\frac{n D}{\varepsilon}\right)\right). 
\]
\end{proof}

\subsection*{Acknowledgment}
This research has been financially supported by The Analytical Center for the Government of the Russian Federation (Agreement No. 70-2021-00143 01.11.2021, IGK 000000D730324P540002)

\end{document}